\documentclass[a4paper,12pt]{amsart}

  \usepackage{amssymb,amsthm}
\usepackage{graphicx,color}    

\usepackage{url} 

\newtheorem{theorem}{Theorem}[section]
\newtheorem{corollary}[theorem]{Corollary}

\theoremstyle{definition}

\def\r{\mathbb R}
\def\h{\mathbb H}

\def\e{\mathbf e}

\def\s{\mathbb S}
\def\M{\mathbb M}
\begin{document}

\title{A note on helices in the Euclidean space and in the hyperbolic space}
\author{Rafael L\'opez}
\address{ Department of Geometry and Topology. University of Granada. 18071  Granada, Spain}
\email{rcamino@ugr.es}
\keywords{general helix, helix, Lancret relation, hyperbolic space, Killing vector field}
  \subjclass{53A04, 53B25}
\dedicatory{Dedicated to Professor Manolo Barros  on the occasion of his retirement from the University of Granada}

\begin{abstract} 
We   prove that general helices in Euclidean space for Killing vector fields associated to rotations are helices,  that is, curves with constant curvature and constant torsion. In hyperbolic space $\h^3$, we obtain the parametrization of helices for the Killing vector fields associated to hyperbolic rotations, spherical rotations and parabolic rotations. It is proved that these helices are geodesics in suitable surfaces of $\h^3$. 

\end{abstract}

\maketitle 

\section{Introduction and motivation}

A {\it general helix} in the Euclidean space $\r^3$ is a curve which makes a constant angle with a fixed direction of $\r^3$. This direction is called the axis of the helix. In  1802, M. A. Lancret asserted that a curve $\gamma$ is a general helix if and only if   the ratio   $\tau/\kappa$ between the torsion $\tau$ and the curvature $\kappa$ of $\gamma$ is constant \cite{la}. This characterization of the general helices is known as the Lancret relation. In fact, this property was later  proved    in 1845 by   Saint Venant \cite{sv}.   A special subclass of general helices are  those ones with constant curvature and constant torsion. These curves are called   {\it helices} and explicit parametrization are known \cite{gra}. These helices are curves contained in circular cylinders of $\r^3$ and they are geodesics in these cylinders. 

The generalization of the notion of helix in other ambient spaces requires to extend the analogue of `fixed direction'. In $\r^3$, a direction is represented by a Killing vector field of $\r^3$ corresponding to the translations of $\r^3$. Therefore, a natural choice is to consider   a `fixed direction' as a Killing vector field.  In 1997, Manuel Barros extended the notion of helices to space forms $\M^3(c)$ by using the notion of a Killing vector field along a curve  \cite{ba}. A  curve $\gamma(s)$ in $\M^3(c)$ is said to be  a general helix if there is a Killing vector field $V$ with constant length  along $\gamma$ such that the angle between $V(s)$ and $\gamma'(s)$ is a non-zero constant along $\gamma$. The vector field $V$ is called the axis of $\gamma$. A property of the space forms $\M^3(c)$ is that   a Killing vector field along $\gamma$ is  the restriction to $\gamma$ of a Killing vector field of $\M^3(c)$. With this generalization, Barros proved the following characterization of the general helices in the space forms $\M^3(c)$ which extends the Lancret relation. We distinguish the three space forms $\M^3(c)$.
\begin{enumerate}
\item  Case $c=0$. A curve $\gamma$ in the Euclidean space $\r^3$ is a general helix   if and only if $\gamma$ is planar or there exists a constant $a\in\r$ such that $\tau=a\kappa$ (Thm. 2 in \cite{ba}).
\item Case $c=-1$. A curve $\gamma$ in the hyperbolic space $\h^3$ is a general helix if and only if $\gamma$ is contained in a geodesic surface or $\gamma$ is a helix (Thm. 1 in \cite{ba}).
\item Case $c=1$. A curve $\gamma$ in the sphere $\s^3$ is a general helix if and only if $\gamma$ is contained in a totally geodesic surface of there exists a constant $a\in\r$ such that $\tau=a\kappa\pm 1$ (Thm. 3 in \cite{ba}).
\end{enumerate}
Thus in $\r^3$ and $\s^3$, we have a Lancret relation for general helices for all Killing vector fields where, in general, $\kappa$ and $\tau$ are functions on the arc-length parameter. In $\h^3$, the Lancret relation is trivial because $\kappa$ and $\tau$ are constant.

The purpose of this paper is to deep in the general helices of $\r^3$ and $\h^3$ and investigate in some of the ideas that appears  in \cite{ba}. A first observation is that in Euclidean space there are more Killing vector fields besides the constant vectors (translations of $\r^3$). The dimension of Killing vector fields in $\r^3$ is $6$ and they are generated by three constant vectors fields, namely, $\partial_x$, $\partial_y$ and $\partial_z$ in $(x,y,z)$ coordinates of $\r^3$ together three Killing vector fields associated to rotations of $\r^3$, $-y\partial_x+x\partial_y$, $-y\partial_z+y\partial_z$ and $-x\partial_z+z\partial_x$. As far the author knows, the study of general helices with these  axes have not been considered in the literature.  In Sect. \ref{s2} we prove that for these vector fields, general helices are in fact helices, that is, curves with constant curvature and constant torsion. 

 \begin{theorem} \label{te0}
A curve $\gamma$ in  Euclidean space $\r^3$ is a general helix for a Killing vector field  of rotations if and only  $\gamma$ is a helix. 
\end{theorem} 

From a similar viewpoint, Munteanu and Nistor investigated     curves in Euclidean space $\r^3$ that make a constant angle with these Killing vector fields of rotations \cite{mn}. However, they do not assume that the Killing vector field has constant length along the curve. Thus, and according to Thm. \ref{te0}, these curves  are not, in general, helices and   they may do not satisfy the Lancret relation.

The second scenario is the hyperbolic space $\h^3$. Again, after the Barros' paper, there is not literature on the general helices of $\h^3$ which is surprising  in a classic topic in differential geometry as it is the study of helices. This may be due that   the general helices are helices and the interest may be reduced.   Even in such a situation, it deserves to know the parametrizations of these helices.   Although this may seem a mere exercise, the dimension of the Killing vectors fields of $\h^3$ is $6$. The variety of Killing vector fields in $\h^3$   requires a work in finding all types of helices for each one of the Killing vector fields.  In this paper, we find in Sect. \ref{s3} the explicit parametrizations of the helices for all Killing vector fields that come from uniparametric groups of rigid motions of $\h^3$.  See Theorem \ref{t31} (hyperbolic rotations), Theorem \ref{t32} (spherical rotations) and Theorem \ref{t33} (parabolic rotations). In consequence, it remains as an open question  the computation of the helices for the other two Killing vector fields of $\h^3$.

Other notions of helices have been considered in the literature for different choices of `fixed direction' and ambient spaces. Without to included here all references, we refer the following: in Minkowski space \cite{fgl}; with $F$-constant vector field \cite{ly};  in a Riemannian space, using parallel transported vector fields \cite{da,ha};  in a Lie group with a bi-invariant metric and  the direction is a left-invariant vector field \cite{ci}; in Berger spheres and in the product space $\s^2\times\r$  where the fixed direction is the vertical  Killing vector field \cite{mo,ni}.

\section{Helices in Euclidean space associated to rotations}\label{s2}

Consider in Euclidean space $\r^3$ a Killing vector field associated to a uniparametric group of rotations. We  prove that general helices for this axis are  helices, that is, curves with constant curvature and constant torsion. In other words, the Lancret relation $\tau=a\kappa$ for non-constant functions $\tau$ and $\kappa$ only holds for Killing vector fields associate to translations.  This implies that the family of general helices for the Killing vector fields of rotations is very small comparing with the helices whose axis is a constant Killing vector field. 

Without loss of generality, we can assume that the Killing vector field is associated to   the rotations about the $z$-axis. Then 
\begin{equation}\label{ww}
W=-y\partial_x+x\partial_y.
\end{equation}
We now prove Thm. \ref{te0} obtaining the parametrizations of the corresponding helices.

\begin{theorem} \label{te01}
A curve $\gamma$ in  Euclidean space $\r^3$ is a general helix for the Killing vector field $W$   if and only $\gamma$   is the $z$-axis    of equation $\{x=0,y=0\}$  or $\gamma$ is a helix. The parametrizations of these helices are
\begin{equation}\label{h2}
\gamma(s)=(r\cos\frac{\cos\theta}{r} s,r\sin \frac{\cos\theta}{r} s,s\sin\theta),
\end{equation}
where $r>0$ and $\theta$ is the angle that makes $\gamma$ with the axis $W$.
\end{theorem} 

\begin{proof} 
Let $\gamma(s)=(x(s),y(s),y(s))$, $s\in I$, be a curve in $\r^3$ parametrized by arc-length and suppose that $\gamma$ is a general helix with axis $W$.  Since $W$ has constant length along $\gamma$, there is a constant $r>0$ such that $|W(\gamma(s))|=r$. Thus $x(s)^2+y(s)^2=r^2$ for all $s\in I$. If $r=0$, then $\gamma$ is the vertical line of equation $x=y=0$. The angle is $\theta=\pi/2$. This proves the first case of the theorem.

Suppose now $r>0$. This implies that $\gamma$ is contained in the circular cylinder $x^2+y^2=r^2$ of radius $r$   whose axis is the $z$-axis. Then there is a function $\varphi=\varphi(s)$ such that $x(s)=r\cos\varphi(s)$ and $y(s)=r\sin\varphi(s)$.    If $\theta$ is the angle between $\gamma$ and $W$, then 
$$\cos\theta=\frac{\langle \gamma'(s),W(\gamma(s))\rangle}{|W(\gamma(s))|}=\frac{xy'-yx'}{r}=r\varphi'(s).$$
This yields $\varphi(s)=\frac{\cos\theta}{r} s$. Moreover, $z'^2=1-x'^2-y'^2=1-\cos^2\theta$. Then, up to a sign, 
we have $z(s)=s\sin\theta$. This implies that $\gamma$ is parametrized by \eqref{h2}.   Its curvature and torsion are
$$\kappa=\frac{\cos^2\theta}{r},\quad \tau=\frac{\sin\theta\cos\theta}{r}.$$
\end{proof}

Let us observe that   the general helices of the constant Killing vector field include planar curves. These curves have zero torsion, but arbitrary curvature. However this is not the case for the helices of $W$. The only helix with $\tau=0$ is when $\theta=0$. The helix is the horizontal circle $s\mapsto (r\cos(s/r),r\in(s/r),c)$ and its curvature is $\kappa=1/r$. 

Notice that the helix \eqref{h2} is the standard helix of axis $\partial_z$. It is obvious that the radius is $r$. The pitch $h$ is calculated as the height difference of two consecutive points with the same $xy$-projection. It is not difficult to see that the pitch is obtained as the height of the points $s=0$ and $s=2\pi /\cos\theta$, obtaining that $h=2\pi r \tan\theta$.

As a consequence, we obtain the surprising property that the helices \eqref{h2} are also helices for   the Killing vector field $\partial_z$, that is, the same curve is helix for two different axes.

\begin{corollary} General helices of $\r^3$ with axis $ -y\partial_x+x\partial_y$   coincide with the family of  helices of axis $\partial_z$. 
\end{corollary}

\begin{proof}   General helices of $ -y\partial_x+x\partial_y$   are parametrized by   \eqref{h2}, which are helices of $\partial_z$. The converse is immediate.  
\end{proof}

We point out  that if $\theta$ is the angle that makes $\gamma$ with the axis $W$, the angle that makes $\gamma$ with $\partial_z$ is $\frac{\pi}{2}-\theta$.

We finish this section comparing the results of Thm. \ref{te01} with the curves that make a constant angle with the vector field $W$. As we said in the Introduction, these curves were obtained in \cite{mn}. The authors in \cite{mn} do not assume that the   Killing vector field $W$ has constant length along the curve. Recall that this condition implies that   the curve is contained in a circular cylinder. If we drop this condition, then the potential curves are in arbitrary position in $\r^3$. In \cite{mn}, the authors investigated this situation obtaining a full classification of this family of curves. Therefore, and as a conclusion of Thm. \ref{te01} these curves  do not satisfy, in general, the Lancret relation $\tau=a\kappa$, otherwise, the curvature and the torsion would be constant. 

We now recall the construction given in \cite{mn}. The curves with constant angle with the Killing vector field $W$ are parametrized by   
\begin{equation}\label{gg} 
\gamma(s)=\left(r(s)\cos\phi(s),\ r(s)\sin\phi(s),z(s)\right).
\end{equation}
 The parameter $s$ is the arc length parameter of $\gamma$ and $r=r(s)$, $z=z(s)$ and $\phi=\phi(s)$   are smooth  functions  defined by  
\begin{equation*}
\begin{split}
r(s)&=r_0+\sin\theta\int\limits^s \cos\omega(t) dt,\\
 z(s)&=\sin\theta\int\limits^s \sin\omega(t) dt,\\
\phi(s)&=\cos\theta\int\limits^s\frac{dt}{r(t)},
\end{split}
\end{equation*}
where $r_0\in{\mathbb{R}}$ and $\omega=\omega(s)$ is a smooth function on $I\subset\r$. Here we notice a gap in \cite{mn} where the circular helices \eqref{h2} do not appear. However, if we take $\omega$ to be constant and   $\omega=\pi/2$, then we obtain from \eqref{gg} the curve $\gamma(s)=(r_0 \cos \left(\frac{s\cos\theta}{r_0}\right),r_0 \sin \left(\frac{s\sin\theta}{r_0}\right),s \sin \theta )$. This curve coincides with  the curve \eqref{h2}.

We now show two explicit and interesting examples for some choices of $\omega$ in \eqref{gg}. These examples will allow to understand well the context of the Barros' results.
\begin{enumerate}
\item Suppose that  $\omega$ is a constant function,  $\omega(s)=w$ and $w\not=\pi/2$. In this case,  the curve \eqref{gg}   is contained in the cone $x^2+y^2=\cot^2(w) z^2$.A computation of the curvature and torsion of the corresponding curve \eqref{gg} gives 
\begin{equation*}
\begin{split}
\kappa(s)&=\frac{\sqrt{4 \csc ^2(\theta )-2 \cos (2 \theta ) \sin ^2(w)+\cos (2 w)-5}}{2 s \cos (w)},\\
\tau(s)&=\frac{\cos (\theta ) \tan (w)}{s}.
\end{split}
\end{equation*}
In consequence $\tau/\kappa$ is constant and thus, according to the Barros' result, the curve is a general helix with its corresponding axis. But since the curvature and torsion are not constant, Theorem \ref{te01} implies that its axis does not come from rotations of $\r^3$. Then necessarily its axis is a constant Killing vector field.    In order to find the axis, we apply the Barros' result which says that the axis is $\cos\varphi T(s)+\sin\varphi N(s)$, where $\varphi$ is the angle that makes the helix with its axis. A priori, we do not know what is the angle $\varphi$. A simple computation proves that by taking $\varphi$ as $\cos\varphi=\sin\theta\sin(w)$, then the axis is $(0,0,1)=\partial_z$. In fact, we have  $\langle\gamma'(s),\partial_z\rangle=\cos\varphi$. Definitively, the curve $\gamma$ has the following properties: it is a general helix for $\partial_z$, it is not a helix for the Killing vector field  $W$ and $\gamma$ makes a constant angle with $\partial_z$ as well as with the Killing vector field $W$.
 
\item Let $\omega(s)=s$ in \eqref{gg}. Then the curve $\gamma$, \eqref{gg}, is parametrized by  
$$\gamma(s)=\left(\begin{array}{l}
\sin \theta \sin s \cos \left(\cot \theta  \log \left(\tan \left(\frac{s}{2}\right)\right)\right)\\
\sin \theta  \sin s \sin \left(\cot \theta  \log \left(\tan \left(\frac{s}{2}\right)\right)\right)\\
-\sin \theta\cos s\end{array}
\right).$$
The curve $\gamma$ is contained in a sphere centered at $0$ and radius $|\sin\theta|$. Its curvature and torsion are
$$\kappa=\frac{\sqrt{1-\sin ^2(\theta ) \cos ^2(s)}}{\sin (\theta ) \sin (s)},\quad \tau=\frac{\cos (\theta )}{1-\sin ^2(\theta ) \cos ^2(s)}.$$
Therefore this curve does not satisfy the Lancret relation. However, this curve makes a constant angle with the Killing vector field $W$, \eqref{ww}.
\end{enumerate}

  \section{Helices in hyperbolic space}\label{s3}
  In this section we investigate the helices of $\h^3$ whose axis is a Killing vector field given by a uniparametric family of group of rotations.    Consider the upper half space model of the hyperbolic space $\h^3$, that is,   $\r^{3}_+=\{(x,y,z)\in\r^{3}\colon z>0\}$, endowed with the metric
$$\langle\cdot,\cdot\rangle=\frac{dx^2+dy^2+dz^2}{z^2}.$$
 Denote by $\{\e_1,\e_2,\e_3\}$ to the canonical basis of $\r^3$. In this section, the Euclidean metric of $\r^3_{+}$ is denoted by $\langle\cdot,\cdot\rangle_e$. Since   the metric $\langle\cdot,\cdot\rangle$ is conformal to the Euclidean metric $\langle\cdot,\cdot\rangle_e$  of $\r^3_{+}$, the Levi-Civita connections  $\nabla$ of $\h^3$  and $\nabla^e$ of  $\r^3$ are related by
\begin{equation}\label{r}
\nabla_XY=\nabla^e_XY-\frac{X_3}{z}Y-\frac{Y_3}{z}X+\frac{\langle X,Y\rangle_e}{z}\e_3,
\end{equation}
for two vector fields $X$ and $Y$. The  subindex $()_3$ represents the coordinate of the vector field with respect to the basis $\{\partial_x,\partial_y,\partial_z\}$.

In $\h^3$,   the space of Killing vector fields is $6$-dimensional. A basis is formed by 
$$\{ x\partial_x+y\partial_y+z\partial_z\, -y\partial_x+x\partial_y, \partial_x, \partial_y, \frac{x^2}{2}\partial_x+xy\partial_y+xz\partial_z,  xy\partial_x+\frac{y^2}{2}\partial_y+yz\partial_z \}.$$
The first four Killing vector fields are associated to a uniparametric groups rigid motions of $\h^3$. We recall these groups.
\begin{enumerate}
\item The Killing vector field $V=x\partial_x+y\partial_y+z\partial_z$  is tangent to the hyperbolic rotations $\{\mathcal{H}_t\colon t\in\r\}$ from $0\in\r^3$. From the Euclidean viewpoint, these rotations are dilations from $0$, $\mathcal{H}_t(x,y,z)=e^t(x,y,z)$.
\item The Killing vector field $W=-y\partial_x+x\partial_y$ is tangent to the elliptic rotations $\{\mathcal{E}_t\colon t\in\r\}$. From the Euclidean viewpoint, the rotations about the $z$-axis are $\mathcal{E}_t(x,y,z)=(x\cos t-y\sin t,x\sin t+y\cos t,z)$.
\item The Killing vector field $P=\partial_x$ is tangent to the parabolic rotations $\{\mathcal{P}_t\colon t\in\r\}$ which are, from the Euclidean viewpoint, horizontal translations along the $x$-axis, $\mathcal{P}_t(x,y,z)=(x+ t,y,z)$. The Killing vector field $\partial_y$ is also tangent to parabolic rotations changing the $x$-axis by the $y$-axis.
\end{enumerate}

Let $\gamma(s)$ be a non-geodesic curve in $\h^3$ parametrized by arc-length with   Frenet frame   $\{T,N,B\}$. The Frenet equations are  
\begin{equation*}
\begin{split}
\nabla_TT&=\kappa N,\\
\nabla_TN&=-\kappa T+\tau B,\\
\nabla_TB&=-\tau B,
\end{split}
\end{equation*}
where $\kappa$ and $\tau$ are the curvature and the torsion  of $\gamma$, respectively. 

In the following three subsections, we calculate the helices of each one of the above four vector Killing fields. Since the study with the Killing vector field $\partial_y$ is similar to $\partial_x$, we only do with $\partial_x$. As a consequence of the results, we point out the following observation. In principle, in the definition of general helix, the angle that makes the curve with the axis is arbitrary, but it may occur that not all angles were possible. However, we will prove that, indeed, there exist helices for the full range of angles. The same   happened in Euclidean space with the general helices regardless the type of axis.

  \subsection{The Killing vector field $V=x\partial_x+y\partial_y+z\partial_z$}

  \begin{theorem}\label{t31}
   Helices with axis the Killing vector field $V=x\partial_x+y\partial_y+z\partial_z$ and making an angle $\theta$ are the $z$-axis of equation $\{x=0,y=0\}$ or the curves parametrized by 
  \begin{equation}\label{h0}
 \gamma(s)=me^{\frac{\cos\theta}{\sqrt{1+c^2}}s}\left(c \cos\frac{s\sin\theta}{c},c \sin \frac{s\sin\theta}{c}, 1\right).
 \end{equation}
where $m,c>0$. The curvature $\kappa$ and the torsion $\tau$ of $\gamma$   are
\begin{equation}\label{h01}
\kappa=\frac{c^2+\sin^2\theta}{c\sqrt{1+c^2}},\quad \tau=\frac{\sin\theta\cos\theta}{c\sqrt{1+c^2}}.
\end{equation}
Moreover these helices are contained in the Killing cylinder of equation $x^2+y^2=c^2z^2$ and they are geodesics in this cylinder.
  \end{theorem}
  \begin{proof}
    Let $\gamma(s)$, $s\in I$, be a curve such that $V$ is constant along $\gamma$. If $\gamma(s)=(x(s),y(s),z(s))$, then there is a constant $c\geq 0$  such that 
  $$\frac{x^2+y^2+z^2}{z^2}=1+c^2.$$
  If $c=0$, then $\gamma$ is the geodesic given by the vertical line $\{x=0,y=0\}$. This curve makes an angle $\theta=0$ with $V$. Notice that $V(s)={\partial_z}_{|\gamma(s)}$.
  
  From now, we suppose $c>0$. Then  $\gamma$ is contained in the Euclidean  cone $\mathcal{C}_c$ of equation $x^2+y^2=c^2 z^2$. In hyperbolic geometry, the surface $\mathcal{C}_c$ is a Killing cylinder  because it is the locus of points equidistant from the geodesic $\{x=0,y=0\}$.
  
  Since $\gamma$ is contained in $\mathcal{C}_c$,   there exists a smooth function $\varphi=\varphi(s)$ such that 
 \begin{equation}\label{p1}
 \gamma(s)=(cz(s)\cos\varphi(s),cz(s)\sin\varphi(s),z(s)).
 \end{equation}
 Let $\theta$ be the angle between $\gamma$ and $V$. The modulus of $V$ is obtained from 
 $$|V|^2=\langle \gamma,\gamma\rangle=\frac{x^2+y^2+z^2}{z^2}=1+c^2.$$
 Then
 $$\cos\theta=\frac{\langle V,\gamma'\rangle}{|V(\gamma)|}=\frac{xx'+yy'+zz'}{\sqrt{1+c^2}z^2}=\sqrt{1+c^2}\frac{ z'(s)}{z(s)}.$$
 By solving this ODE,  there is $m>0$ such that
 $$z(s)=me^{\frac{\cos\theta}{\sqrt{1+c^2}}s}.$$
 Using that $\gamma$ is parametrized by arc-length, then $x'^2+y'^2+z'^2-z^2=0$. This gives 
 $\varphi'(s)=\frac{\sin\theta}{c}$. Thus, up to a constant on the parameter $s$, we have  $\varphi(s)=\frac{\sin\theta}{c}s$. Definitively, this defines the curve $\gamma$  is given by \eqref{h0}.
 
  We compute the curvature and the torsion of $\gamma$. First, we calculated the normal vector $N$ of $\gamma$. By using \eqref{r}, we have 
\begin{equation}\label{normal}
\begin{split}
\nabla_TT&=\nabla^e_TT-2\frac{T_3}{z}T+\frac{\langle T,T\rangle_e}{z}\e_3\\
&=\frac{m(c^2+\sin^2\theta)}{1+c^2}e^{\frac{\cos\theta}{\sqrt{1+c^2}}s}\left(-\frac{1 }{  c}\cos \left(\frac{s \sin (\theta )}{c}\right),\frac{ - 1}{ c}\sin \left(\frac{s \sin (\theta )}{c}\right) , 1 \right). 
\end{split}
\end{equation}
Computing the modulus, we obtain $\kappa$ given in \eqref{h01}. Hence we have  the unit normal vector by $N=\nabla_TT/k$. Now the computation of the torsion given in \eqref{h01} is obtained by the calculation of  $\nabla_TN$. 

In order to prove that the helices \eqref{h0} are geodesics  in the cylinder $\mathcal{C}_c$ it is enough to check that the normal vector $N$ of $\gamma$ coincides with that unit normal vector $G$ of $\mathcal{C}_c$ along $\gamma$. The cylinder $\mathcal{C}_c$ is given by the implicit equation $x^2+y^2-c^2z^2=0$. Since the hyperbolic metric is conformal to the Euclidean one, the vector $G$ is proportional to the Euclidean gradient $\nabla^e$ of the function $f(x,y,z)= x^2+y^2-c^2z^2$. Then 
\begin{equation*}
\begin{split}
\nabla^ef_{|\gamma}&=2(x,y,-c^2 z)_{|\gamma}\\
&=2me^{\frac{\cos\theta}{\sqrt{1+c^2}}s}\left(c \cos\frac{s\sin\theta}{c},c \sin \frac{s\sin\theta}{c}, -c^2\right)\\
\end{split}
\end{equation*}
This vector is proportional to $\nabla_TT$, \eqref{normal}, and thus it is proportional to $N$.
  \end{proof}
According to \cite{ba}, the axis of the helix \eqref{h0} is $\frac{V(s)}{|V(s)}|=\cos\theta T(s)+\sin\theta B(s)$ which can be easily checked.

 Helices of the Killing vector field $V$ are related with the curves in Euclidean space $\r^3$ that make a constant angle with the position vector. These curves were studied in \cite{bo}. Since the Euclidean and the hyperbolic metrics are conformal in the upper half space model of $\h^3$, these curves are candidates to be the helices \eqref{h0} that appear in Thm. \ref{t31}. In order to give a precise relation between both curves, we recall these curves. 
  
 For any $r,h\in\r$, $r>0$, define in Euclidean space $\r^3$the curve $\gamma_{r,h}$ by 
 $$\gamma_{r,h}(t)=e^t(r\cos t,r\sin t, h).$$
 This curve appears in \cite[p. 28]{bo} proving that $\gamma_{r,h}$ makes a constant angle with the position vector $\gamma_{r,h}(t)$ for all $t\in \r$. This curve is contained in the cone $\mathcal{C}_{r/h}$. We   see that viewed in hyperbolic space, the curve $\gamma_{r,h}$ makes a constant angle with $V$ and we also compute its curvature and torsion as curve of $\h^3$.  
 
 First, we need to parametrize $\gamma_{r,h}$ by the hyperbolic arc-length. An easy computation gives that this parametrization is  
 \begin{equation}\label{h1}
 \gamma_{r,h}(s)=e^{\phi(s)}(r\cos\phi(s),r\sin\phi(s),h),\quad \phi(s)=\frac{hs}{\sqrt{2r^2+h^2}}.
 \end{equation}
 We now calculate the angle between the curve $\gamma_{r,h}$ and $V$. We have 
  $$
  \cos\theta=\frac{\langle V,T\rangle}{|V|}=\frac{\sqrt{r^2+h^2}}{\sqrt{2r^2+h^2}}.$$
  Then $\sin\theta=\frac{r}{\sqrt{2r^2+h^2}}$. Notice that $|V|=\frac{\sqrt{r^2+h^2}}{h}$
  
 We compute the curvature and the torsion of $\gamma_{r,h}$ as curve in $\h^3$. By using the relation \eqref{r}, we have 
\begin{equation*}
\begin{split}
\nabla_TT&=\nabla^e_TT-2\frac{T_3}{\gamma_3}T+\frac{\langle T,T\rangle_e}{\gamma_3}\e_3\\
&=e^{\phi}\frac{2 h^2 r}{ 2 r^2+h^2}\left(-  \cos \phi,- \sin \phi,1\right)
\end{split}
\end{equation*}
  Hence 
  $$\kappa=|\nabla_TT|=\frac{2r\sqrt{r^2+h^2}}{2r^2+h^2}.$$
  The unit normal vector is  
  $$N=e^{\phi}\frac{  h^2}{ \sqrt{r^2+h^2}}\left(-  \cos \phi,- \sin \phi,1\right).$$
  To calculate the binormal vector, we know
  \begin{equation*}
\begin{split}
\tau B&=\nabla_TN+\kappa T\\
  &=\frac{h^3 e^{\phi}}{\sqrt{h^2+r^2} \left(h^2+2 r^2\right)^{3/2}}\left(  (h^2+r^2) \sin \phi+r^2 \cos \phi ,-(h^2+r^2) \cos \phi+r^2 \sin \phi , h  r  \right).
  \end{split}
  \end{equation*}
    Then the torsion is obtained by $\tau=\langle \tau B,N\rangle$, which yields
  $$\tau=\frac{h^2}{2r^2+h^2}.$$
  One can easily see that $\gamma_{r,h}$ is a helix with axis $V$ by checking the identity
  $$\frac{V}{|V|}=\cos\theta T+\sin\theta B.$$

  \subsection{The Killing vector field $W= -y\partial_x+x\partial_y$}

 \begin{theorem} \label{t32}
 Helices with axis the Killing vector field $W=-y\partial_x+x\partial_y$  and making an angle $\theta$ are the $z$-axis  of equation $\{x=0,y=0\}$ and the curves parametrized by
  \begin{equation}\label{w0}
 \gamma(s)=me^{\frac{\sin\theta}{\sqrt{1+c^2}}s}\left(c \cos\frac{s\cos\theta}{c},c \sin \frac{s\cos\theta}{c}, 1\right),
 \end{equation}
where $m,c>0$.  The curvature $\kappa$ and the torsion $\tau$ of $\gamma$ are
\begin{equation}\label{h00}
\kappa=\frac{c^2+\cos^2\theta}{c\sqrt{1+c^2}},\quad \tau=-\frac{\sin\theta\cos\theta}{c\sqrt{1+c^2}}.
\end{equation}
Moreover these helices are contained in the Killing cylinder $\mathcal{C}_c$ of equation $x^2+y^2=c^2z^2$ and they are geodesics in this cylinder.
 \end{theorem}
 
 \begin{proof}  
  Let $\gamma(s)=(x(s),y(s),z(s))$, $s\in I$, be a curve such that $V$ has  constant length along $\gamma$. Then there is $c\geq 0$ such that
  $\frac{x(s)^2+y(s)^2}{z(s)^2}=c^2$. If $c=0$, then $x(s)=y(s)=0$ for all $s\in I$. Then $\gamma$ is the $z$-axis. This curve makes an angle $\theta=\pi/2$ with $W$. 
  
  From now, we suppose that $c>0$. Thus $\gamma$ is contained in the cone $\mathcal{C}_c$ of equation  $x^2+y^2=c^2z^2$. Thus we can parametrize $\gamma$ as \eqref{p1}. Let $\theta$ be the angle that makes $\gamma$ with $W$. Notice that $|W|^2=\frac{x^2+y^2}{z^2}=c^2$. Then 
  $$\cos\theta=\frac{\langle W,\gamma'\rangle}{|W(\gamma)|}=\frac{-x'y+xy'}{cz^2}=c\varphi'.$$
  Thus, up to a constant in the parameter $s$, we have
  $$\varphi(s)=\frac{\cos\theta}{c}s.$$
  Using that $\gamma$ is parametrized by arc-length, and the parametrization \eqref{p1}, we have 
  $$c^2(z'^2+\varphi'^2 z^2)+z'^2=z^2.$$
  Since $\varphi'=\cos\theta/c$, this gives
  $$\frac{z'}{z}=\frac{\sin\theta}{\sqrt{1+c^2}}.$$
  Integrating this ODE by separation of variables, we deduce the existence of $m>0$ such that  
  $$z(s)=m e^{\frac{\sin\theta}{\sqrt{1+c^2}}s}.$$
  Using \eqref{p1} again,  the curve $\gamma$ is parametrized by \eqref{h0} where we replace $\theta$ in \eqref{h0} by $\theta-\frac{\pi}{2}$.
  
  Since these curves \eqref{w0} coincide with the helices \eqref{h0},  they are geodesics in the Killing cylinder  $\mathcal{C}_c$.
  \end{proof}

 As in the Euclidean case, the curves \eqref{w0} are helices for two different axes  of $\h^3$.
 
 \begin{corollary} In $\h^3$,   helices of a Killing vector field $x\partial_x+y\partial_y+z\partial_z$   coincide with the   helices of the Killing vector field $-y\partial_x+x\partial_y$. The angles of the helix with both axes are complimentary. 
  \end{corollary}
 
  \subsection{The Killing vector field $P= \partial_x$}

 For $c>0$, let introduce the notation $\mathcal{H}_c=\{(x,y,c)\colon x,y\in\r\}$. This is the horizontal plane of equation $z=c$. In $\h^3$, the set $\mathcal{H}_c$ is a horosphere. This surface is totally umbilical and its mean curvature is $1$.

  \begin{theorem}\label{t33}
   Helices with axis $P=\partial_x$ and making an angle $\theta$ are parametrized by 
    \begin{equation}\label{h3}
  \gamma(s)=(s\cos\theta+x_0,s\sin\theta+y_0,c),
  \end{equation}
  where $x_0, y_0,c\in\r$, $c>0$.   The curvature and torsion of $\gamma$ is $\kappa=1$ and $\tau=0$, respectively.  Moreover, these helices are geodesics in the horosphere $\mathcal{H}_c$.  
\end{theorem}

\begin{proof}
   Let $\gamma=\gamma(s)$, $s\in I$, be a curve in $\h^3$ such that $P$ has constant length along $\gamma$. If $\gamma(s)=(x(s),y(s),z(s))$, then there is $c>0$ such that 
  $$\frac{1}{c^2}=\langle P(\gamma(s)),P(\gamma(s))\rangle=\frac{\langle P(\gamma(s)),P(\gamma(s))\rangle_e}{z(s)^2}=\frac{1}{z(s)^2}.$$
  This implies that the curve $\gamma$ is contained in the horizontal plane $\mathcal{H}_{c}$, that is,  $\gamma$ is included in the horosphere $\mathcal{H}_c$.

  Let $\gamma(s)=(x(s),y(s),c)$ and  suppose that $\gamma$ is parametrized by arc-length. In particular, $x'^2+y'^2=c^2$. If $\theta$ is the angle that makes $\gamma$ with $P$, and because $|P(\gamma(s))|=1/c$, we have 
  $$\cos\theta= \langle\gamma'(s),P\rangle =\frac{\langle\gamma'(s),P\rangle_e}{c^2|P|}=\frac{x'}{c}.$$
  Thus there are $x_0,y_0\in\r$ such that $x(s)=s c \cos\theta+x_0$ and $y(s)=sc \sin\theta+y_0$. This gives the parametrization \eqref{h3}.

  We calculate the Frenet frame as well as  its curvature and torsion. Using \eqref{r}, we have 
  $$\nabla_TT=-2\frac{T_3}{\gamma_3}T+\frac{\e_3}{\gamma_3}=c\,\e_3.$$
Since    $|\nabla_TT|=1$, we obtain $\kappa=1$ and $N(s)=c\,\e_3$. For the calculation of the torsion, we compute $\nabla_TN$, obtaining 
$$
  \nabla_TN=N'-\frac{N_3}{\gamma_3}T-\frac{T_3}{\gamma_3}N=-T.$$
This proves that $\tau=0$ and the binormal vector is $B=c(\sin\theta,-\cos\theta,0)$. 
 
 Finally,      the normal vector of $\gamma$ is $N=c\e_3$ of $\gamma$ which coincides with the unit normal vector of the horosphere $\mathcal{H}_c$. This implies that  $\gamma$ is a geodesic in $\mathcal{H}_c$. 
\end{proof}
 
 Again, it is easy to check  $\partial_x=\cos\theta T+\sin\theta B$. On the other hand, these helices have   zero torsion. According to \cite[Thm. 1]{ba}, the curve must be  contained in a totally geodesic plane of $\h^3$. Indeed, the helix   \eqref{h3} is contained in the vertical plane of equation $\sin\theta x-\cos\theta y=y_0\cos\theta-x_0\sin\theta$ which it is a totally geodesic plane of $\h^3$.

\section*{Data Availability}

This article did not generate or analyze any datasets; therefore, data sharing is not relevant or applicable to this study.


\section*{Acknowledgements}
The author  has been partially supported by MINECO/MICINN/FEDER grant no. PID2023-150727NB-I00,  and by the ``Mar\'{\i}a de Maeztu'' Excellence Unit IMAG, reference CEX2020-001105- M, funded by MCINN/AEI/10.13039/ 501100011033/ CEX2020-001105-M.

\end{document}